\documentclass[reqno]{amsart}
\usepackage{aliascnt}
\usepackage{amsfonts}
\usepackage{mathrsfs}
\usepackage{amsmath}
\usepackage{amsmath,graphics}
\usepackage{amssymb}
\usepackage{indentfirst,latexsym,bm,amsmath,pstricks,amssymb,
amsthm,graphicx,fancyhdr,float,color}
\usepackage[all]{xy}
\usepackage{amsthm}
\usepackage{amscd}
\usepackage{hyperref}
\usepackage[all]{hypcap}
\newcommand{\A}{\mathbb{A}}

\newcommand{\FF}{\mathbb{F}}
\newcommand{\GG}{\mathbb{G}}
\newcommand{\NN}{\mathbb{N}}

\newcommand{\fL}{\mathfrak{L}}

\newcommand{\cM}{\mathcal{M}}

\newcommand{\cT}{\mathcal{T}}

\newcommand{\cX}{\mathcal{X}}

\newcommand{\fa}{\mathfrak{a}}

\newcommand{\fb}{\mathfrak{b}}
\newcommand{\fc}{\mathfrak{c}}

\newcommand{\fg}{\mathfrak{g}}

\newcommand{\fh}{\mathfrak{h}}

\newcommand{\fn}{\mathfrak{n}}

\newcommand{\fs}{\mathfrak{s}}

\newcommand{\ft}{\mathfrak{t}}

\newcommand{\fz}{\mathfrak{z}}

\newcommand{\p}{\partial}

\DeclareMathOperator{\Aut}{Aut}

\DeclareMathOperator{\Char}{char}
\DeclareMathOperator{\Cent}{Cent}
\DeclareMathOperator{\Der}{Der}

\DeclareMathOperator{\Ens}{Ens}

\DeclareMathOperator{\GL}{GL}
\DeclareMathOperator{\Gr}{Gr}

\DeclareMathOperator{\Hom}{Hom}
\DeclareMathOperator{\id}{id}
\DeclareMathOperator{\Lie}{Lie}
\DeclareMathOperator{\Nor}{Nor}

\DeclareMathOperator{\rk}{rk}

\DeclareMathOperator{\U}{U}
\DeclareMathOperator{\Stab}{Stab}

\DeclareMathOperator{\Spec}{Spec}

\DeclareMathOperator{\Tor}{Tor}

\theoremstyle{plain}
\numberwithin{equation}{section}
\newtheorem{Theorem}{Theorem}[section]
\newtheorem{Lemma}[Theorem]{Lemma}
\newtheorem{Corollary}[Theorem]{Corollary}
\newtheorem{Proposition}[Theorem]{Proposition}
\newtheorem{Question}[Theorem]{Question}

\theoremstyle{Theorem}

\theoremstyle{remark}
\newtheorem*{Remark}{Remark}
\newtheorem*{Remarks}{Remarks}
\newtheorem*{Definition}{Definition}

\numberwithin{equation}{section}
\begin{document}
\title{Toral stabilizers of restricted Lie algebras}
\author{Hao Chang}
\address{Mathematisches Seminar, Christian-Albrechts-Universit\"at zu Kiel, Ludewig-Meyn-Str. 4, 24098 Kiel, Germany}
\email{chang@math.uni-kiel.de}
\date{\today}
\begin{abstract}
This paper investigates the toral stabilizer of a finite-dimensional restricted Lie
algebra $(\fg,[p])$, defined over an algebraically closed field $k$ of positive characteristic.
We compute the group of toral stabilizers of some simple Lie algebras and
give a new characterization of the solvable restricted Lie algebras. Moreover,
we determine the subalgebra structure of exceptional simple Lie algebras in good characteristic.
We also provide applications, which concern the distribution of weight spaces.
\end{abstract}
\maketitle
\setcounter{tocdepth}{1}
\tableofcontents
\bigskip

\section{Introduction}
Let $(\fg,[p])$ be a finite-dimensional restricted Lie algebra.
In \cite{Fa04}, Farnsteiner introduced a finite group $S(\fg)$ associated to $\fg$,
which is called toral stabilizer.
The structure group $S(\fg)$ coincides with the Weyl group in case
of classical type or the Lie algebra of a smooth group.
In \cite{BFS}, the authors defined the Weyl group for arbitrary restricted Lie algebra
and established a connected between the Weyl group and toral stabilizer.
In particular, they computed the toral stabilizers of the Lie algebras of Cartan type.

In this paper, we first compute the group of toral stabilizers for some simple Lie algebras.
The toral stabilizer of restricted generalized Witt algebra is determined (see Theorem \ref{S(g)Zassenhaus}).
We give a new characterization of the solvable restricted
Lie algebras and prove that a finite-dimensional restricted Lie algebra $\fg$ is solvable if and only if
the group $S(\fg)$ is a $p$-group (Theorem \ref{maintheoremsolvable}).
The main ingredient of the proof is the classification of absolute toral rank two simple Lie algebras due to
Premet and Strade (see \cite{PS3} and \cite{S2}).

Let $\fg$ be a exceptional simple classical Lie algebra.
In Section \ref{section subalgebra}, we give a criteria for determining
the subalgebra structure of $\fg$.
Recently, Herpel and Stewart proved that only the Witt algebra can occur
as a simple subalgebra of $\fg$ (see \cite{HS}).
Our approach is completely
different by using toral stabilizer to analyse the subgroup of Weyl groups.

In the last, we consider a question of Skryabin (see \cite{Skr15}).
We study the weight space decomposition.
Skryabin's question can be reduced to consider the torus
of maximal dimension (see Proposition \ref{proposition 1}).
We answer the question in some particular cases.
Also, some known results are reproved in a simple way.


\section{Schemes of tori and toral stabilizers}\label{Section 1}
Throughout, we shall be working over an algebraically closed field
$k$ of characteristic $\Char(k)=p>0$.
All vector spaces are assumed to be finite-dimensional.

Given a restricted Lie algebra $(\fg,[p])$,
let $M_k$ and $\text{Ens}$ be the categories of commutative $k$-algebras and sets, respectively.
For a torus $\ft\subseteq\fg$, we consider the scheme $\cT_\fg:M_k\rightarrow\Ens$, given by
$$\cT_\fg(R):=\{\varphi\in\Hom_p(\ft\otimes_k R,\fg\otimes_k R);~\varphi~\text{ is a split injective }R\text{-linear map}\}$$
for every $R\in M_k$. Here $\Hom_p(\ft\otimes_k R,\fg\otimes_k R)$
denotes the set of homomorphisms of restricted $R$-Lie algebras and $\fg\otimes_k R$ carries
the natural structure of restricted $R$-Lie algebra via
$$[x\otimes r,y\otimes s]:=[x,y]\otimes rs,~(x\otimes r)^{[p]}:=x^{[p]}\otimes r^p,~\forall x,y\in\fg, r,s\in R.$$
According to \cite[(1.4)]{FV01}, $\cT_\fg$ is a smooth,
affine algebraic scheme.
Consequently, the connected components of $\cT_\fg$ are irreducible,
and there exists precisely one irreducible component $\cX_{\ft}(k)\subseteq\cT_\fg(k)$
such that the given embedding $\ft\hookrightarrow\fg$
belongs to $\cX_{\ft}(k)$.

We denote by $\mu(\fg)$ and $\rk(\fg)$ the maximal dimension of all
tori $\ft\subseteq\fg$ and the minimal dimension of all Cartan subalgebras $\fh\subseteq\fg$,
respectively.
According to \cite[(7.4)]{Fa04},
the set
$$\Tor(\fg):=\{\ft\subseteq\fg;~\ft~\text{torus}, \dim_k\ft=\mu(\fg)\}$$
of tori of maximal dimension is locally closed within the Grassmannian $\Gr_{\mu(\fg)}(\fg)$ of $\fg$.
Thanks to \cite[(1.6)]{FV01} and \cite[(3.5)(9.3)]{Fa04},
the variety $\Tor(\fg)$ is a smooth irreducible affine variety of
dimension $\dim\Tor(\fg)=\dim_k\fg-\rk(\fg)$.

We consider the scheme $\cT_\fg$ in case $\ft\in\Tor(\fg)$.
Then the automorphism groups $\Aut_p(\fg)$ and $\Aut_p(\ft)$ naturally act on $\cT_\fg(k)$ via
$$g.\varphi:=g\circ\varphi~\text{and}~h.\varphi:=\varphi\circ h^{-1},~\forall \varphi\in\cT_\fg(k),g\in\Aut_p(\fg),h\in\Aut_p(\ft)$$
respectively. Both actions commute, and $\Aut_p(\ft)\cong\GL_{\mu(\fg)}(\mathbb{F}_p)$.
Following \cite{Fa04}, we let
$$S(\fg,\ft):=\Stab_{\Aut_p(\ft)}(\cX_{\ft}(k))$$
be the stabilizer of component $\cX_{\ft}(k)$ in $\Aut_p(\ft)$.
Then $S(\fg,\ft)$ is called the toral stabilizer of $\fg$ (relative to $\ft$).
If $\ft'\in\Tor(\fg)$ is another torus, then there exists an isomorphism
$h:\ft\rightarrow\ft'$ such that $S(\fg,\ft')=hS(\fg,\ft)h^{-1}$
(see \cite[Theorem 4.1]{Fa04} and \cite[Theorem 1.1]{BFS}).
Hence we set $$S(\fg):=S(\fg,\ft)$$
be the toral stabilizer of $\fg$.

From now on we assume $\fh\unlhd \fg$ to be a $p$-ideal of $\fg$,
and let $\ft\in\Tor(\fg)$.
Then $\ft\cap\fh$ and $(\ft+\fh)/\fh$ are tori of maximal dimension of $\fh$ and $\fg/\fh$,
respectively (see for example \cite[Proposition 2.3]{BFS}).
Also, we can find a subtorus $\ft'\subseteq\ft$ such that $\ft=(\ft\cap\fh)\oplus\ft'$ (see \cite[Page 1347]{BFS}).
Furthermore,
a $p$-ideal $\fh\unlhd\fg$ is called elementary abelian if $[\fh,\fh]=\{0\}=\fh^{[p]}$.

For future reference, we recall the following results, cf.\cite[(5.6)(5.12)]{Fa04} and \cite[(2.1)(2.3)]{BFS}:

\begin{Lemma}\label{lemma1}
Let $\fh\subseteq\fg$ be a $p$-subalgebra of the restricted Lie algebra $(\fg,[p])$ and $\ft\in\Tor(\fg)$.
\begin{enumerate}
\item[(1)]
If $\fh\unlhd\fg$ is an elementary abelian ideal,
then we have $S(\fg,\ft)\cong S(\fg/\fh,\ft)$.
\item[(2)] If $\fh\unlhd\fg$ is a toral $p$-ideal,
then there is a surjective map $S(\fg,\ft)\rightarrow S(\fg/\fh,\fs)$,
where $\ft=\fh\oplus\fs$.
\item[(3)] If $\mu(\fh)=\mu(\fg)$, then $S(\fh,\ft)\subseteq S(\fg,\ft)$ is a subgroup.
\item[(4)] If $\fh\unlhd\fg$ is a $p$-ideal, then there is an injective homorphism
$$S(\fg,\ft)\hookrightarrow\left(\begin{array}{cc}
S(\fh,\ft\cap\fh)&\Lie_p(\ft',\ft\cap\fh)\\
0&S(\fg/\fh,\ft')\\
\end{array}\right),$$
where $\ft=(\ft\cap\fh)\oplus\ft'$.
\end{enumerate}
\end{Lemma}

\begin{Lemma}\label{lemma2}
Let $\fh\subseteq\fg$ be a $p$-subalgebra of the restricted Lie algebra $(\fg,[p])$.
If there exists a torus $\ft\in\Tor(\fg)$
such that $\ft\cap\fh\in\Tor(\fh)$,
then there is an injective homomorphism
$$S(\fh,\ft\cap\fh)\hookrightarrow S(\fg,\ft).$$
\end{Lemma}
\begin{proof}
As before, we can choose a subtorus $\ft'\subseteq\ft$ with $\ft=(\ft\cap\fh)\oplus\ft'$.
The canonical inclusion
$$\iota:\fh\hookrightarrow\fg$$
induces a morphism
$$\iota_{*}:\cT_\fh(k)\rightarrow \cT_{\fg}(k);\varphi\mapsto\varphi\oplus\id_{\ft'},$$
which gives rise to an embedding
$$S(\fh,\ft\cap\fh)\hookrightarrow S(\fg,\ft);g\mapsto g\oplus\id_{\ft'}.$$
\end{proof}

\begin{Remark}
The Lemma \ref{lemma1}(3) is a special case of Lemma \ref{lemma2}.
\end{Remark}

The reader is referred to \cite{SF} and \cite{S1} for basic
facts concerning the structure theory of simple Lie algebras.

We denote by $W(n), S(n), H(n);n=2r$ and $K(n);n=2r+1$ the restricted
simple Lie algebras of Cartan type.
\begin{Lemma}\cite[Theorem 5.3]{BFS}.\label{S(g)cartan type}
Let $\fg$ be a Lie algebra of Cartan type ($W,S,H,K$).
Then there is an isomorphism
$$S(\fg)\cong\GL_{\mu(\fg)}(\FF_p).$$
\end{Lemma}

We mention the case of the restricted Melikian algebra $\cM(1,1)$.
Recall that this is a simple graded Lie algebra over a field
of characteristic $p=5$, which is neither classical nor of Cartan type.
We refer the reader to \cite{Skr} for the details.
Note that $\mu(\cM(1,1))=2$ (\cite[Corollary 4.4]{Skr}).

\begin{Lemma}\cite[Theorem 5.5]{BFS}.\label{S(g)melikian}
Assume that $k$ has characteristic $p=5$.
Let $\cM(1,1)$ be the restricted Melikian algebra.
Then there is an isomorphism
$$S(\cM(1,1))\cong\GL_2(\FF_p).$$
\end{Lemma}

\begin{Lemma}\label{toralstablizerofh21phitau}
Suppose that $p>3$.
Let $\fL:= H(2;\underline{1};\Phi(\tau))^{(1)}$ be the simple non-graded Hamiltonian Lie algebra and
$\fg:=\fL_p$ the minimal $p$-envelope of $\fL$.
Then there is an isomorphism
$$S(\fg)\cong\GL_2(\FF_p).$$
\end{Lemma}
\begin{proof}
Recall that $\fL\subseteq W(2,\underline{1})$ is a subalgebra (see \cite[Page 41]{S2}).
Hence we can regard $\fg$ as a restricted subalgebra of $ W(2,\underline{1})$.
Note that $\mu(\fg)=2$ (see \cite[Page 2]{S2}).
According to \cite[Corollary 10.3.3]{S2} and \cite[Theorem 1.5]{BFS},
there is an isomorphism
$$\Nor_G(\ft)/\Cent_G(\ft)\cong S(\fg,\ft),$$
where $G$ is the automorphism group of $(\fg,[p])$.
Moreover, it follows from \cite[Theorem 7.3.2]{S1} that $G$ is a subgroup of $\Aut(W(2,\underline{1}))$.
Thanks to \cite[Lemma 2]{P92}, the group $\Cent_G(\ft)$ is trivial.
Hence it suffices to determine the group $\Nor_G(\ft)$.
It follows from the proof of \cite[Corollary 10.3.3]{S2} that,
for any $g\in \GL_2(\FF_p)$,
we can find some $\tilde{g}\in\Nor_G(\ft)$ such that $\tilde{g}|_\ft=g$, as required.
\end{proof}

\section{Generalized Witt algebra}

In this section, we turn to the computation of the toral stabilizer of the generalized Witt algebra $W(m,\underline{n})$ (The reader is referred to \cite{Ree} or \cite{W1},
\cite[\S 4.2]{SF} and \cite[7.6]{S1} for details). We assume that $p>3$.

We denote by $\mathfrak{L}:=W(m,\underline{n})$ the genelized Witt algebra, or also Witt algebra.
Let $A(m,\underline{n})$ be the divided power algebra in $m$ indeterminates.
By definition, $\mathfrak{L}$ consists of all special derivations of $A(m,\underline{n})$, where $\underline{n}\in\NN^m$ is a multi-index (see \cite[Page 146]{SF}).
It is a free $A(m,\underline{n})$-module generated by the partial derivatives $\partial_1,\dots,\partial_m$.
The length of $\underline{n}$ is defined to be $|\underline{n}|=n_1+\cdots+n_m$.
Let
$$A(|\underline{n}|,\underline{1}):=k[X_1,\cdots,X_{|\underline{n}|}]/(X_1^p,\cdots,X_{|\underline{n}|}^p)$$ be the truncated polynomial ring of $|\underline{n}|$ indeterminates.
Then $$W(|\underline{n}|,\underline{1}):=\Der(A(|\underline{n}|,\underline{1}))$$ is the $|\underline{n}|$-th Jacobson-Witt algebra.

Also, define the following distinguished multi-indices: for all $k\in\{1,\dots,m\}$, the $k$-th unit multi-index is $\epsilon_k=(\delta_{1,k},\dots,\delta_{m,k})$. One can check that the assignment:
$$(\forall i\in\{1,\dots,m\}),~(\forall j\in\{0,\dots,n_i-1\});~y_{i,j}\mapsto x^{(p^j\epsilon_i)}$$
extends uniquely to an isomorphism:
\begin{equation}
\phi: A(m,\underline{n})\cong k[\{y_{i,j}\}_{i,j}]\cong A(|\underline{n}|,\underline{1})
\end{equation}

Consequently, this results an isomorphism of Lie algebras:
\begin{equation}\label{1}
\Phi:\Der A(m,\underline{n})\rightarrow \Der A(|\underline{n}|;\underline{1})=W(|\underline{n}|,\underline{1});~D\mapsto \Phi(D),
\end{equation}
where $\Phi(D)(f)=\phi(D(\phi^{-1}(f))),~\forall f\in A(|\underline{n}|,\underline{1})$.

Let $\mathfrak{L}_p$ be the minimal $p$-envelop of $\fL$.
We call it restricted generalized Witt algebra.
In fact, $\fL_p$ coincides with the derivation algebra
$\Der W(m,\underline{n})$ (see \cite[Theorem 7.2.2]{S1}).
By a direct check, $\Phi$ is a restricted isomorphism.
Hence, we can identify $\fL_p$ as a restricted subalgebra of $W(|\underline{n}|,\underline{1})$.
Therefore, we have the following embedding map:
\begin{equation}\label{inclusion}
\iota:=\Phi|_{\fL_p}:\fL_p\hookrightarrow W(|\underline{n}|,\underline{1}).
\end{equation}
We denote by $G:=\Aut_p(\fL_p)$ (respectively $\tilde{G}:=\Aut_p(W(|\underline{n}|,\underline{1}))$)
the automorphism group of $\fL_p$ (respectively $W(|\underline{n}|,\underline{1})$).
It is well-known that $\tilde{G}$ is connected algebraic group. Moreover,
$G$ is a connected subgroup of $\tilde{G}$ (\cite{W1},\cite{W}).

Let $(\fg,[p])$ be a restricted Lie algebra, $G:=\Aut_p(\fg)^{\circ}$
be its connected automorphism group.
Following \cite{BFS}, a torus $\ft\in\Tor(\fg)$ is called generic
if the orbit $G.\ft$ is a dense subset of $\Tor(\fg)$.

Note that $\mu(\fL_p)=\mu(W(|\underline{n}|,\underline{1}))=|\underline{n}|$ (see \cite[Theorem 7.6.3]{S1} and \cite[Theorem 1]{D1}).
According to \cite[Theorem 7.6.2]{S1},
there exists a torus $\ft_0\in\Tor(\fL_p)$ such that $\iota(\ft_0)$
is a generic torus of $W(|\underline{n}|,\underline{1})$ (cf. \cite[Theorem 3.2]{BFS}).
By the map (\ref{inclusion}),
we have $\Tor(\fL_p)\subseteq\Tor(W(|\underline{n}|,\underline{1}))$.

\begin{Lemma}\label{generic torus of Zassenhaus}
Keep the notations as above.
Let $\fL_p:=W(m,\underline{n})_p$ be the restricted genelized Witt algebra.
Then $\ft_0$ is a generic torus of $\fL_p$.
\end{Lemma}
\begin{proof}
By virtue of \cite[Lemma 7.3]{Fa04}, it is easy to see that
$\Tor(\fL_p)\subseteq\Tor(W(|\underline{n}|,\underline{1}))$ is a closed subvariety.
By the above , we know that
$$\overline{\tilde{G}.\ft_0}=\Tor(W(|\underline{n}|,\underline{1})).$$
Since both of them are irreducible varieties (see Section \ref{Section 1}),
it is enough to show that
$$\tilde{G}.\ft_0\cap\Tor(\fL_p)=G.\ft_0,$$
where $G:=\Aut_p(\fL_p)$ is the automorphism group of $\fL_p$.
As $G\subseteq\tilde{G}$, one side is obvious.
We need to prove that
$$\tilde{G}.\ft_0\cap\Tor(\fL_p)\subseteq G.\ft_0.$$
We observe that, setting
$$S(i):=\sum_{\alpha<p^{n_i}}kx_i^{(\alpha)}\partial_i\cong W(1,\underline{n_i});~~1\leq i\leq m.$$

Thanks to \cite[Theorem 7.6.2]{S1} and the proof \cite[Theorem 7.6.3(2)]{S1}, we can assume that $\ft_0$ is generated by the elements of the form $$\partial_i+u_i\partial_i,~u_i\partial_i\in S(i)_{(0)},~1\leq i\leq m,$$
where $S(i)_{(0)}$ is the filtration $0$ part of $S(i)$.
Suppose that there exists an element $\tilde{g}\in\tilde{G}$ such that $\tilde{g}.\ft_0\in\Tor(\fL_p)$.
For any $1\leq i\leq m$,
we say that $$\tilde{g}.(\partial_i+u_i\partial_i)=\sum_{s=1}^{m}f_{i,s}\partial_s,$$
where $f_{i,s}\in A(m,\underline{n})$.
By (\ref{1}), we know that there exists an element $\varphi\in\Aut(A(|\underline{n}|,1)$
such that the corresponding automorphism is $\tilde{g}$. Thus,
\begin{equation}\label{2}
\sum_{s=1}^{m}\phi(f_{i,s})\mathcal{D}_s=\varphi(\phi(1+u_i))\tilde{g}(\mathcal{D}_i),
\end{equation}
where $\mathcal{D}_i=\iota(\partial_i)=D_{i,0}+\sum\limits_{j=1}^{n_i-1}(-1)^{j}y_{i,0}^{p-1}\cdots y_{i,j-1}^{p-1}D_{i,j}$ and $D_{i,j}$ is the partial derivation with respect to $y_{i,j}$.
Suppose that
$$\tilde{g}(\mathcal{D}_i)=\sum\limits_{k,l}a_{k,l}^iD_{k,l},~a_{k,l}^i\in A(|\underline{n}|,\underline{1}).$$
By comparing the coefficients of $D_{k,l}$ in (\ref{2}),
we obtain:
\begin{equation}\label{3}
a_{s,0}^i\varphi(\phi(1+u_i))=\phi(f_{i,s}).
\end{equation}
\begin{equation}\label{4}
a_{s,j}^i\varphi(\phi(1+u_i))=\phi(f_{i,s})(-1)^{j}y_{s,0}^{p-1}\cdots y_{s,j-1}^{p-1};~1\leq j\leq n_s-1.
\end{equation}
Since $\varphi(\phi(1+u_i))$ is invertible in $A(|\underline{n}|,\underline{1})$,
it follows from (\ref{3}) and (\ref{4}) that
\begin{equation}
a_{s,j}^i=(-1)^{j}y_{s,0}^{p-1}\cdots y_{s,j-1}^{p-1}a_{s,0}^i;~1\leq j\leq n_s-1.
\end{equation}
Consequently, 
$$\tilde{g}(\iota(\partial_i))=\sum_{s=1}^{m}a_{s,0}^i\iota(\partial_s).$$
Hence,
$\tilde{g}\in G$ (cf. \cite[Theorem 12.8]{Ree}, \cite[Theorem 2]{W1}).
Last, it should be notice that the intersection $\tilde{G}.\ft_0\cap\Tor(\fL_p)$
is a non-empty open set of $\Tor(\fL_p)$.
Hence, the orbit $G.\ft_0$ is dense in $\Tor(\fL_p)$ and $\ft_0$ is a generic torus of $\fL_p$.
\end{proof}

The main result of this section reads:

\begin{Theorem}\label{S(g)Zassenhaus}
Let $\fL_p:=W(m,\underline{n})_p$ be the restricted genelized Witt algebra.
Then there is an isomorphism
$$S(\fL_p)\cong\GL_{|\underline{n}|}(\FF_p).$$
\end{Theorem}
\begin{proof}
Let $G:=\Aut_p(\fL_p)$ be the connected automorphism of $\fL_p$.
Since $\ft_0$ is a generic torus of $\fL_p$ (Lemma \ref{generic torus of Zassenhaus}),
it follows from \cite[Theorem 1.5]{BFS} that there is an
isomorphism
$$S(\fL_p,\ft_0)\cong\Nor_G(\ft_0)/\Cent_G(\ft_0).$$
Note that $G$ is a subgroup of $\tilde{G}$.
It follows from \cite[Lemma 2]{P92} that the centralizer $\Cent_{\tilde{G}}(\ft_0)$ is a trivial group,
so is $\Cent_{G}(\ft_0)$.
Thus, there is a natural inclusion map
\begin{equation}\label{5}
\Nor_G(\ft_0)\hookrightarrow\Nor_{\tilde{G}}(\ft_0).
\end{equation}
We claim that (\ref{5}) is also surjective.
Let $\tilde{g}\in\Nor_{\tilde{G}}(\ft_0)$,
so that $\tilde{g}.\ft_0=\ft_0\in\Tor(\fL_p)$.
Now the proof of Lemma \ref{generic torus of Zassenhaus}
ensures that $\tilde{g}\in G$, as desired.

Hence,
$$S(\fL_p)\cong\Nor_G(\ft_0)=\Nor_{\tilde{G}}(\ft_0)\cong\GL_{|\underline{n}|}(\FF_p),$$
where the last isomorphism follows from \cite[Theorem 1(i)]{P92} (see also Lemma \ref{S(g)cartan type}).
\end{proof}

\section{Solvable Lie algebra}\label{sectionmainthm}
In this section, we would like to give a new characterization of the solvable restricted Lie algebras.

\begin{Theorem}\cite[Theorem 6.2]{Fa04}.\label{rolftheorem}
Let $(\fg,[p])$ be a solvable,
restricted Lie algebra.
Then $S(\fg)$ is a $p$-group.
\end{Theorem}

Recall that the classification of finite-dimensional simple Lie algebras
of absolute toral rank two over an algebraically closed field of characteristic
$p>3$ (see \cite[Theorem, Page 2]{S2}, also \cite[Theorem 1.1]{PS3}).

\begin{Theorem}\label{classification}
Let $\fL$ be a finite-dimensional simple Lie algebra
of absolute toral rank two over an algebraically closed field of characteristic
$p>3$. Then $\fL$ is isomorphic to one of the Lie algebras listed below:
\begin{enumerate}
\item[(I)] classical type $A_2,B_2$ or $G_2$.
\item[(II)] the restricted Lie algebras $W(2;\underline{1}), S(3;\underline{1})^{(1)},H(4;\underline{1})^{(1)},K(3,\underline{1})$.
\item[(III)] non-restricted Cartan type Lie algebras:
\begin{enumerate}
\item[(a)] $W(1;\underline{2}),H(2;\underline{1};\Phi(1))$;
\item[(b)] $H(2;(1,2))^{(2)}$;
\item[(c)] $H(2;\underline{1};\Phi(\tau))^{(1)}$,
\end{enumerate}
\item[(IV)] the Melikian algebra $\cM(1,1),p=5$.
\end{enumerate}
\end{Theorem}

\begin{Theorem}\label{maintheoremsolvable}
Suppose that $p>3$.
Let $(\fg,[p])$ be a finite-dimensional restricted Lie algebra.
Then $\fg$ is solvable if and only if the toral stabilizer $S(\fg)$ is a $p$-group.
\end{Theorem}
\begin{proof}
One implication follows from Theorem \ref{rolftheorem},
so that the converse has to be proved.

Suppose there exists a non-solvable restricted Lie algebra $\fg$
such that $S(\fg)$ is a $p$-group.
Furthermore, we assume that $\fg$ has minimal dimension.

First, we claim that $\fg$ is semisimple.
Let $\fa\unlhd\fg$ be a non-zero abelian $p$-ideal of minimal dimension.
Then $\fa^{[p]}\subseteq\fz(\fg)$ is also a $p$-ideal of $\fg$,
where $\fz(\fg)$ is the center of $\fg$. Hence, we have $\fa^{[p]}=0$
or $\fa^{[p]}=\fa$, so that $\fa$ is an elementary abelian ideal or
a toral $p$-ideal. Thanks to Lemma \ref{lemma1},
the toral stabilizer $S(\fg/\fa)$ is a $p$-group.
By induction, $\fg/\fa$ is solvable,
it follows that $\fg$ is solvable. This is a contradiction,
thus, $\fg$ is semisimple.

Furthermore, we claim that $\fg$ has no non-zero proper $p$-ideal.
Let $\fn\unlhd\fg$ be a proper $p$-ideal and $\ft\in\Tor(\fg)$.
We know that
the $p$-subalgebra $\ft\cap\fn$ is a torus of $\fn$ of maximal dimension,
i.e., $\ft\cap\fn\in\Tor(\fn)$.
Thanks to Lemma \ref{lemma2},
we have $S(\fn)$ is a $p$-group.
By the induction hypothesis, this implies that $\fn$
is a solvable ideal. We arrive at a contradiction since we have shown that $\fg$ is semisimple.

Thanks to \cite[Lemma 6.3]{FSW},
the derived subalgebra $[\fg,\fg]$ of $\fg$ is a simple Lie algebra and $\fg$ is a minimal $p$-envelope of $[\fg,\fg]$.
For convenience, from now on, we denote by $\fL:=[\fg,\fg]$ and $\fg=\fL_p$.

We first discuss the case that $\mu(\fg)=1$.
By virtue of \cite[Theorem 10.6.1]{S2},
$\fL$ is isomorphic to one of
$$\mathfrak{sl}_2, W(1;\underline{1}), H(2;\underline{1})^{(2)}.$$
According to Lemma \ref{S(g)cartan type} (see also the proof of \cite[Corollary 6.6]{Fa04}),
the group $S(\fg)$ is not a $p$-group,
which is a contradiction.

Let $\fh\subseteq\fg$ be any $2$-section of $\fg$ with respect to a torus $\ft\in\Tor(\fg)$ (cf.\cite[Definition 1.3.9]{S1}).
By definition, $\ft\subseteq\fh$, we have $\mu(\fh)=\mu(\fg)$.
It follows from Lemma \ref{lemma1}(3) that $S(\fh)\subseteq S(\fg)$
is a subgroup, so that $S(\fh)$ is a $p$-group.
The minimality implies that $\fg=\fh$ or $\fh$ is solvable.
If all $2$-sections of $\fg$ are solvable, then $\fg$ is solvable (see \cite[Theorem 1.3.10]{S1}).
Consequently, $\fg=\fh$ and $\mu(\fg)=2$ (\cite[Theorem 1.3.11(2)]{S1}).

According to the classification result of simple Lie algebras by Premet and Strade (Theorem \ref{classification}).
We proceed by a case-by-case analysis.
Note that if $\fL$ is restricted, then $\fg=\fL$.
In these case,
we can conclude that $S(\fg)$ is not a $p$-group
(the proof of \cite[Corollary 6.6]{Fa04} for classical type, Lemma \ref{S(g)cartan type}
for Cartan type and Lemma \ref{S(g)melikian} for Melikian algebra).

Assume that $\fL\cong W(1;\underline{2})$.
Then there exists a torus $\ft\in\Tor(\fg)$ such that
every $1$-section is isomorphic to the restricted
Witt algebra $W(1;\underline{1})$ (see \cite[Theorem 7.6.2]{S1}).
It follows from Lemma \ref{lemma2} that
$S(W(1;\underline{1}))$ can be identified with a subgroup of $S(\fg)$.
Since $S(W(1;\underline{1}))=\GL_1(\FF_p)=\FF_p^{*}$ (Lemma \ref{S(g)cartan type}),
this ensures that $S(\fg)$ is not a $p$-group.
We obtain a contradiction.

Now, we consider the case $\fL\cong H(2;(1,2))^{(2)}$.
Let $\ft\in\Tor(\fg)$.
It follows from \cite[Lemma 10.2.3]{S2} that $\dim_k(\ft\cap\fL)=1$
and $\fg=\ft+\fL$. Furthermore, there exists a proper $p$-subalgebra
$\fL_{(0)}$ and a torus $\ft\in\Tor(\fg)$ such that $\ft\cap\fL_{(0)}\in\Tor(\fL_{(0)})$ (see \cite[Page 39]{S2}).
Now Lemma \ref{lemma2} implies that the toral stabilizer $S(\fL_{(0)})$ is a $p$-group.
By the induction, $\fL_{(0)}$ is solvable.
We obtain a contradiction since there is a subalgebra $\fL_{(1)}\subseteq \fL_{(0)}$ such that
$\fL_{(0)}/\fL_{(1)}\cong\mathfrak{sl}_2$ (see the proof of \cite[Lemma 10.2.3]{S2}).
If $\fL\cong H(2;\underline{1};\Phi(1))$,
then the proof is similar to the case $\fL\cong H(2;(1,2))^{(2)}$ (see \cite[Theorem 10.4.6]{S2} and \cite[Theorem 16.2.5]{S3}).

Thus we must have $\fL\cong H(2;\underline{1};\Phi(\tau))^{(1)}$.
According to Lemma \ref{toralstablizerofh21phitau},
the toral stabilizer $S(\fg)$ is not a $p$-group, which is a contradiction.
\end{proof}

\begin{Corollary}
Let $(\fg,[p])$ be a finite-dimensional restricted Lie algebra over a field of
prime characteristic $p>3$, and let $\fn\unlhd\fg$ be a $p$-ideal.
Then $S(\fg)$ is a $p$-group if, and only if,  both $S(\fg/\fn)$ and $S(\fg)$ are
$p$-groups.
\end{Corollary}
\begin{proof}
This is a direct consequence of Lemma \ref{lemma1}(4) and Theorem \ref{maintheoremsolvable}.
\end{proof}

\begin{Corollary}\label{affine solvable}
Suppose that $p>3$.
Let $(\fg,[p])$ be an $l$-dimensional restricted Lie algebra such that
$\cX_\ft(k)\cong\A^{l-\rk(\fg)}$. Then $\fg$ is solvable.
\end{Corollary}
\begin{proof}
According to \cite[Theorem 6.7]{Fa04}, the toral stabilizer is a $p$-group.
Now our assertion follows from Theorem \ref{maintheoremsolvable}.
\end{proof}

\begin{Remarks}
(1). It should be interesting to find a connection between the
representation theory of solvable restricted Lie algebra and the toral stabilizer,
see \cite{FSW} for the maximal $0$-PIM property.

(2). Suppose that $p>3$.
Let $(\fg,[p])$ be an $l$-dimensional restricted Lie algebra with $\mu(\fg)=1$,
$\ft\subseteq\fg$ is a one-dimensional torus.
Hence Theorem \ref{maintheoremsolvable} and Corollary \ref{affine solvable}
imply that the following statements are equivalent:
\begin{enumerate}
\item[(a)] $\fg$ is solvable.
\item[(b)] $S(\fg,\ft)=\{1\}$.
\item[(b)] $\cX_\ft(k)\cong\A^{l-\rk(\fg)}$.
\end{enumerate}

(3). In the proof of Theorem \ref{maintheoremsolvable},
we can also use Theorem \ref{S(g)Zassenhaus} to deal with the case $\fL\cong W(1;\underline{2})$.
\end{Remarks}

\begin{Question}
Suppose that $p>3$.
Determine the solvable restricted Lie algebra $(\fg,[p])$ with trivial
toral stabilizer, i.e., $S(\fg)=\{1\}$.
\end{Question}

\section{Sylow subalgebra}
Suppose that $p>3$.
Let $(\fg,[p])$ be a restricted Lie algebra with toral stabilizer $S(\fg)$.
We call a $p$-subalgebra $\fc$ of $\fg$ a Sylow subalgebra if
$S(\fc)$ is a Sylow $p$-subgroup of $S(\fg)$.
If $\fg$ contains such a Sylow subalgebra $\fc$,
then $\fc$ should be solvable (Theorem \ref{maintheoremsolvable}).

Let we consider the $n$-th Jacobson-Witt algebra $W(n;\underline{1})$.
In this section, we will always assume that $\fg:=W(n;\underline{1})$.

Let $B_n$ be the truncated polynomial ring in $n$ variables:
$$B_n=k[X_1,\cdots,X_n]/(X_1^p,\cdots,X_n^p)\,\,(n\geq 1).$$
By definition, $\fg:=\Der(B_n)$ is the algebra of derivation
of $B_n$. Let $G:=\Aut_p(\fg)$ be the automorphism of $\fg$.
According to \cite[Theorem 2]{W}, $G$ is a connected group and
any automorphism of $\fg$ can be induced from an automorphism of $B_n$.
There results an isomorphism $G\cong\Aut(B_n)$.
Let
$$x_i:= X_i+(X_1^p,\cdots,X_n^p)\in B_n.$$
Denote by $D_{i}=\p/\p x_i\in\fg$ the partial derivative with respect to the
variable $x_i$. Then $\{D_1,\cdots,D_n\}$ is a basis of the
$B_n$-module $\fg$, so that $\dim_{k}\fg=np^n$.
Let $\ft_0:=\langle(1+x_1)D_1,\dots,(1+x_n)D_n\rangle$ be the standard
generic torus of $\fg$ (see \cite[Page 1337]{BFS}).
According to \cite[Theorem 1(i), Lemma 2]{P92},
the Weyl group $W(\fg,\ft_0):=\Nor_G(\ft_0)/\Cent_G(\ft_0)$
relative to $\ft_0$ is isomorphic to the general linear group $\GL_n(\FF_p)$ and $\Cent_G(\ft_0)$ is trivial.
Furthermore, there is an isomorphism $W(\fg,\ft_0)\cong S(\fg)$ (\cite[Proposition 3.3]{BFS}).
Let $\fc$ be a maximal solvable subalgebra which contains $\ft_0$.
Thanks to the result of \cite[Section 4]{Shu}, $\fc$ is unique up to conjugation by $G$.
Recall that $\fg$ has a natural $\mathbb{Z}$-grading structure (see \cite[\S 4.2]{SF}),
$$\fg=\bigoplus\limits_{i=-1}^{n(p-1)-1}\fg_i.$$
We take the standard maximal solvable subalgebra, still denote by $\fc$ (see \cite[Section 3]{Shu}).
Recall that
\begin{equation}\label{standard borel}
\fc=\fg_{-1}+\fb+\sum\limits_{i=1}^{n}\sum\limits_{\alpha(i)}x_1^{\alpha_1}\cdots x_i^{\alpha_i}D_i
\end{equation}
where $\fb$ is the standard Borel subalgebra of $\fg_0\cong\mathfrak{gl}_n$ and
$\alpha(i)=(\alpha_1,\dots,\alpha_i)$ with $|\alpha(i)|>1$ (see \cite[Section 3.1]{Shu} for details).
We denote by $\U\subseteq\GL_n(\FF_p)$ be the lower-triangular unipotent subgroup of $\GL_n(\FF_p)$.
It is well-known (and can be easily verified) that $\U$ is a Sylow $p$-subgroup of $\GL_n(\FF_p)$.

\begin{Lemma}\label{automorphismexist}
Keep the notations as above.
Let $g\in\U$.
Then there exists an element $\tilde{g}\in G$ such that $\fc$ is
$\tilde{g}$-stable and the restriction $\tilde{g}|_{\ft_0}=g$.
\end{Lemma}
\begin{proof}
We write $g=(a_{ij})$.
It follows that $a_{ii}=1,a_{ij}\in\FF_p$ and $a_{ij}=0$ for $i>j$.
Note that the set $\{(1+x_i);~1\leq i\leq n\}$
generate the algebra $B_n$.
We construct an automorphism $\Phi\in\Aut(B_n)$ as follows:
$$\Phi(1+x_i):=\prod\limits_{j=1}^{i}(1+x_j)^{a_{ji}};~~1\leq i\leq n.$$
Let $\tilde{g}\in G$ be the corresponding automorphism.
Then it is easy to check that $\tilde{g}|_{\ft_0}=g$ (Use a formula of Demushkin, see \cite[Page 234]{D1}).
The another assertion follows from the structural observation of
(\ref{standard borel}) and a basic computation.
\end{proof}

\begin{Remark}
In the Lemma \ref{automorphismexist},
the element $\tilde{g}$ is unique (see \cite[Lemma 2]{P92}).
\end{Remark}

\begin{Theorem}\label{sylowsubalgebraofwn}
The standard maximal solvable subalgebra $\fc$ is a Sylow subalgebra of $\fg$.
\end{Theorem}
\begin{proof}
Since $S(\fg)\cong\GL_n(\FF_p)$ (Lemma \ref{S(g)cartan type}), it suffices to prove that $S(\fc)\cong\U$.
Thanks to Lemma \ref{automorphismexist} and \cite[Lemma 1.2]{BFS}, the group $U$ is a subgroup of $S(\fc)$.
In view of Theorem \ref{maintheoremsolvable}, the group $S(\fc)$ is a $p$-group of $\GL_n(\FF_p)$.
Note that $U$ is a Sylow $p$-subgroup of $\GL_n(\FF_p)$.
Thus, they must be equal.
\end{proof}

\begin{Remark}
Suppose that $p>3$.
We consider the special linear Lie algebra $\fg:=\mathfrak{sl}(n)$.
The symmetric group $S_n$ is the Weyl group of $\fg$.
Thanks to \cite[Theorem 4.7]{Fa04}, there is an isomorphism $S(\fg)\cong S_n$. Suppose that $p>n$.
The trivial group is the only Sylow subgroup of $S(\fg)$.
Hence, the Sylow subalgebra always exists (e.g. Let $\fc$ be any Borel subalgebra of $\fg$).
Suppose that $p\leq n$. If there exists a Sylow subalgebra $\fc$ (assume that $\fc$ is maximal),
then maximal solvable subalgebra can not be the Borel subalgebra.
\end{Remark}

Let $\fg$ be a restricted Lie algebra of Cartan type $S(n;\underline{1})^{(1)}$
or $H(2n;\underline{1})^{(2)}$ (see \cite[(4.3)(4.4)]{SF}).
According to \cite[Lemma 4.1 and 4.3]{BFS}, there result embeddings
$$W(\mu(\fg);\underline{1})\hookrightarrow \fg.$$
This yields:
\begin{Proposition}\label{sylowsubalgebrasofSH}
There exists a Sylow subalgebra of $\fg$ when $\fg$ is type $S$ or $H$
\end{Proposition}
\begin{proof}
According to Lemma \ref{S(g)cartan type}, there is an isomorphism
$$S(\fg)\cong\GL_{\mu(\fg)}(\FF_p).$$
Now,
we consider the above embedding map
$$W(\mu(\fg);\underline{1})\hookrightarrow \fg.$$
Hence the image of Sylow subalgebra of $W(\mu(\fg);\underline{1})$ is
a Sylow subalgebra of $\fg$ (Theorem \ref{sylowsubalgebraofwn} implies that such Sylow subalgebra exists).
\end{proof}

To sum up this section, we propose to consider the following question:
\begin{Question}
Suppose that $p>3$.
Whether every restricted Lie algebra has a Sylow subalgebra?
\end{Question}

\section{Subalgebras in the exceptional Lie algebras}\label{section subalgebra}
In this section,
we will explain how to determine the subalgebra structure of the
simple classical Lie algebras of exceptional type by using the toral stabilizer.

Let $G$ be a connected simple algebraic group.
Fix a Borel subgroup $B$ containing some maximal torus $T$
and let $U$ be the unipotent radical of $B$.
Let $$\cX(T):=\Hom(T,\GG_m)$$ be the character group and let
$\Phi\subseteq\cX(T)$ be the root system associated to $G$ with respect to $T$.
Let $\Lambda(\Phi)$ be the integral weight lattices of $\Phi$.
It is clear that $\cX(T)\subseteq\Lambda(\Phi)$.
The quotient $\Lambda(\Phi)/\cX(T)$ is called the fundamental group of $G$ and is denoted by $\pi_1(G)$.
We say that $G$ is simply connected if its fundamental group is trivial.

If $\beta=\sum_im_i\alpha_i$ is the highest root written as a linear combination of simple
roots then $p$ is bad for $\Phi$ if $p=m_i$ for some $i$.
Similarly we may write the dual
of this root as a linear combination of dual simple roots
$\beta^{\vee}=\sum_im_i'\alpha_i^{\vee}$ and $p$ is torsion
for $\Phi$ if $p=m_i'$ for some $i$.
A prime is good (respectively non-torsion) if it is not bad (respectively torsion).

\begin{Definition}
Let $G$ be a connected simple algebraic group.
We say that $p$ is a torsion prime of $G$,
when it is a torsion prime of $\Phi$ or when $p$ divides $\pi_1(G)$.
\end{Definition}
\begin{Remark}
This definition is in fact \cite[Lemma 2.5]{Ste}.
For the original definition of torsion primes of $G$, see \cite[Definition 2.1]{Ste}.
\end{Remark}
\begin{Lemma}\cite[Theorem 3.14]{Ste}.\label{Steinberglemma}
Let $\fg:=\Lie(G)$ be the Lie algebra of a connected simple algebraic $k$-group and $p=\mathrm{char}(k)$.
Then the following statements are equivalent:
\begin{enumerate}
\item[(a)] $p$ is not a torsion prime for $G$.
\item[(b)] The centralizer $C_G(H)$ is connected for every commutative set $H\subseteq\fg$ of semisimple elements.
\end{enumerate}
\end{Lemma}

We denote by $W(G,T)=\Nor_G(T)/\Cent_G(T)$ the Weyl group of $G$
relative to $T$. We write $W=W(G,T)$ for short.

\begin{Lemma}\cite[Theorem 4.7]{Fa04}.\label{rolflemma}
Let $\fg:=\Lie(G)$ be the Lie algebra of a smooth, connected, algebraic group,
$\ft:=\Lie(T)$ a torus of dimension $\mu(\fg)$.
Then the following statements hold:
\begin{enumerate}
\item[(1)] $S(\fg,\ft)\cong\Nor_G(\ft)/\Cent_G(\ft)$.
\item[(2)] If $p\geq 5$, then $S(\fg,\ft)\cong W(G,T)$.
\end{enumerate}
\end{Lemma}

In the following, we assume that $G$ is an exceptional algebraic group.
For future reference, we need the following:
\begin{table}[!htbp]
\centering
\caption{Order of Weyl groups,bad primes and torsion primes for $\Phi$.
see for example \cite[2.13]{GM}}\label{table1}
\begin{tabular}{|c|c|c|c|}
\hline
Type  & Order of $W$    & bad prime& torsion prime \\
\hline
$E_6$ & $2^73^45$       & $2,3$ & $2,3$\\
\hline
$E_7$ & $2^{10}3^457$   & $2,3$ & $2,3$\\
\hline
$E_8$ & $2^{14}3^55^27$ & $2,3,5$ & $2,3,5$\\
\hline
$F_4$ & $2^73^2$        & $2,3$ & $2,3$\\
\hline
$G_2$ & $2^23$          & $2,3$ & $2$\\
\hline
\end{tabular}
\end{table}

\begin{Theorem}\label{maintheorem2}
Let $G$ be a simple, simply connected exceptional algebraic $k$-group with Lie algebra $\fg:=\Lie(G)$,
where $p:=\mathrm{char}(k)$ is a good prime for $G$,
and let $\fh$ be a restricted Lie algebra such that $\mu(\fg)=2$.
If $S(\fh)\cong\GL_2(\FF_p)$, then $\fh$ can not be isomorphic to any restricted subalgebra of $\fg$.
\end{Theorem}
\begin{proof}
Suppose that $\fh$ is isomorphic to some restricted subalgebra of $\fg$.
We still denote by $\fh$.
By assumption, there exists some torus $\ft_0\in\Tor(\fh)$ and $\dim_k\ft_0=2$.
Since $\fg$ is an algebraic Lie algebra, we can find a maximal torus $T\subseteq G$
such that
$$\ft:=\Lie(T)\in\Tor(\fg)$$
and $\ft_0\subseteq\ft$ (cf. \cite[Corollary 4.4]{Fa04}).
Moreover, there is a subtorus $\ft'\subseteq\ft$ such that $\ft=\ft_0\oplus\ft'$ (see \cite[Page 1347]{BFS}).
According to Lemma \ref{lemma2}, we have the following inclusion map:
\begin{equation}\label{111}
\iota:S(\fh,\ft_0)\hookrightarrow S(\fg,\ft);~h\mapsto h\oplus\id_{\ft'}.
\end{equation}

By assumption, $p\geq 5$ (Table \ref{table1}).
Thanks to Lemma \ref{rolflemma}, this results:
\begin{equation}\label{222}
S(\fh,\ft_0)\hookrightarrow S(\fg,\ft)\cong W(G,T)\cong\Nor_G(\ft)/\Cent_G(\ft).
\end{equation}

We denote by $G':=C_G(\ft')$ the centralizer of $\ft'$ in $G$.
Furthermore, $p$ is not a torsion prime for $G$ (See Table \ref{table1}).
It follows from Lemma \ref{Steinberglemma} that $G'$
is a connected reductive subgroup of $G$.
It is clear that $T\subseteq G'$.
We denote by $W':=W(G',T)\cong\Nor_{G'}(\ft)/\Cent_{G'}(\ft)$ the Weyl group of $G'$ (Lemma \ref{rolflemma}).
Then (\ref{111}) is equivalent to the following:
\begin{equation}\label{333}
\iota:S(\fh,\ft_0)\hookrightarrow W'\subseteq W.
\end{equation}
By assumption, we know that $$|S(\fh)|=p(p-1)^2(p+1).$$
On the other hand,
(\ref{222}) implies that $S(\fh)$ can be identified with a subgroup of $W$,
so that the order of Weyl group divided by $|S(\fh)|$.
As one sees from the Table \ref{table1}, we only need to consider the following cases:

\begin{equation}
\fg\cong E_6,E_7;~p=5~~~~~~~~~\text{or}~~~~~~\fg\cong E_7,E_8;~p=7.
\end{equation}

According to (\ref{333}), we know that one of the following two cases must be satisfied:

\begin{enumerate}
\item[(a)] $W'\cong\GL_2(\FF_p)$.
\item[(b)] $W'\cong\GL_2(\FF_p)\ltimes(\FF_p)^{s},~s\in\mathbb{N}\setminus\{0\}$.
\end{enumerate}

By a direct computation, we obtain
\begin{center}
\begin{tabular}{|c|c|c|c|c|}
\hline
Type  & $p=5$~(a)    & $p=5$~(b) & $p=7$~(a)& $p=7$~(b)\\
\hline
$E_6$ & $|W'|=480$ & $|W'|$& &\\
\hline
$E_7$ &$|W'|=480$  & $|W'|$& $|W'|=2016$ &$|W'|$\\
\hline
$E_8$ &  & &$|W'|=2016$ &$|W'|$ \\
\hline
\end{tabular}
\end{center}

Note that $W'$ is a reflection subgroup of $W$.
For the case (a), we will obtain a contradiction from \cite[Section 6, Table 3, 4, 5]{DPR}.
For the case (b), it is obvious that the order of Weyl group $W$ can be divided by $p^2$,
this is also a contradiction (See Table \ref{table1}).
\end{proof}

We retain our assumption that $G$ is an exceptional algebraic $k$-group and $p$ is good for $G$.
Recently, it is proved by Herpel and Stewart (\cite[Theorem 1.3]{HS})
that only the Witt algebra $W(1,\underline{1})$ can occur as a simple subalgebra of non-classical type of $\fg$.
Of course, by simple dimension arguments, the large dimensions of non-classical simple Lie algebras cannot fit
inside the exceptional Lie algebra (see \cite[Section 4]{HS} for details).

\begin{Corollary}[Herpel-Stewart]
Let $G$ be a simple, simply connected exceptional algebraic $k$-group with Lie algebra $\fg:=\Lie(G)$,
where $p:=\mathrm{char}(k)$ is a good prime for $G$.
Let $\fL$ be one of the following simple Lie algebras:
$$H(2;\underline{1};\Phi(\tau))^{(1)},~W(2,\underline{1}),~W(1,\underline{2}),$$
or $$\cM(1,1),~p=5.$$
Then $\fg$ does not contain Lie subalgebras isomorphic to $\fL$.
\end{Corollary}
\begin{proof}
We denote by $\fh:=\fL_p$ the minimal $p$-envelop of $\fL$ (see \cite[Corollary 1.1.8]{S1}).
Suppose that $\fg$ contains $\fL$.
By the same argument as in \cite[Page 787]{HS} (See also \cite[(4,1)(4.3)]{P15}),
we know that $\fg$ contains $\fh$,
and $\fh$ is a restricted subalgebra of $\fg$.

According to Lemma \ref{S(g)cartan type}, \ref{S(g)melikian},
\ref{toralstablizerofh21phitau} and Theorem \ref{S(g)Zassenhaus},
we have $S(\fh)\cong\GL_2(\FF_p)$.
However, this contradicts Theorem \ref{maintheorem2}.
\end{proof}

\begin{Remarks}
(1). Recently, Premet gave an alternative proof for the most difficult case
where $\fL\cong H(2;\underline{1};\Phi)^{(2)}$ (see \cite[Proposition 4.1]{P15}).
Our approach is completely different.

(2). For the case $\fL\cong\cM(1,1)$,
we can also use the dimension argument (See \cite[Lemma 4.1]{HS}).
\end{Remarks}

\section{Weight space decomposition}

In this section, we give some applications,
concern the distribution of weight spaces.

Let $(\fg,[p])$ be a restricted Lie algebra.
By definition, a restricted $\fg$-module $M$ is a $\fg$-module
such that for every $x\in\fg$, the operator $m\mapsto x^{[p]}.m$ is the $p$-th power of the transformation
of $M$ effected by $x$.
If $\ft\subseteq\fg$ is a torus,
then $M$ is a completely reducible $\ft$-module,
give rise the weight space decomposition
$$M=\bigoplus\limits_{\lambda\in\Lambda_M}M_{\lambda}.$$
Each $\lambda\in\Lambda_M$ is a linear form satisfying $\lambda(t^{[p]})=\lambda(t)^p$
for every $t\in\ft$.
As a result, $\Lambda_M\subseteq\Lie_p(\ft,k)$ is a subset of the character group of $\ft$,
i.e., the additive group of all homomorphisms $\ft\rightarrow k$ of restricted Lie algebras,
where the $p$-map on $k$ is the associative $p$-th power.

We first generalize a basic result from \cite{Fa04}.

\begin{Lemma}\label{weight space decomp}
Let $\ft\subseteq\fg$ be a maximal torus. Then there exists some $\varphi\in\cX_\ft(k)$ such that
$\ft':=\varphi(\ft)$ is contained in some torus of maximal dimension.
Moreover, if $M$ is a restricted $\fg$-module with weight space decompositions
$M=\oplus_{\lambda\in\Lambda}M_\lambda$, and $M=\oplus_{\lambda'\in\Lambda'}M_{\lambda'}$,
respectively, then there exists an isomorphism $\xi:\ft\rightarrow\ft'$ such that $\lambda\mapsto\lambda\circ\xi^{-1}$
defines a bijection $\Lambda\rightarrow\Lambda'$ satisfying $\dim_kM_{\lambda\circ\xi^{-1}}=\dim_kM_\lambda$
for every $\lambda\in\Lambda$.
\end{Lemma}
\begin{proof}
We put $A:=k[\cX_\ft]$ and $j:\ft\hookrightarrow\tilde{\fg}$ be
the universal embedding of $\ft$ (see \cite[Page 4186]{Fa04}),
where $\tilde{\fg}:=\fg\otimes_k A$ is a restricted $A$-Lie algebra.
The torus acts on $\tilde{\fg}$ via $j$, we denote by $\tilde{\fg}_0\subseteq\tilde{\fg}$
the corresponding zero weight space.
For any homomorphism $x\in\Spec_k(A)(k)$, let
$$\varphi_x:=(\id_{\fg}\otimes x)\circ j\in\cT_\fg(k)$$
denote the corresponding embedding.
Then
$$\tilde{\fg}_0(x)=(\id_\fg\otimes x)(\tilde{\fg}_0)$$
is the zero weight space relative to the torus $\varphi_x(\ft)$.
According to \cite[Theorem 3.5(3)]{Fa04}, there exists one element $x_0\in\cX_\ft(k)$
such that $\mu(\tilde{\fg}_0(x_0))=\mu(\fg)$.
Note that $\tilde{\fg}_0(x_0)=C_\fg(\varphi_{x_0}(\ft))$.
Let $\ft':=\varphi_{x_0}(\ft)$, so that $\ft'$ is contained in some torus of maximal dimension.

Let $$\tilde{M}=\bigoplus\limits_{\gamma\in\Gamma_{\tilde{M}}}\tilde{M}_{\gamma}$$
be the weight space decomposition of $\tilde{M}:=M\otimes_k A$
induced by the universal embedding $j$.
We consider the the map $\varphi_{x_0}$ and the canonical embedding $\ft\hookrightarrow\fg$.
Then the second assertion follows from the same argument as in \cite[(4.2)]{Fa04}.
\end{proof}

\begin{Lemma}\cite[Corollary 6.2, Corollary 6.3]{BFS}.\label{BFSlemma}
Let $\ft\in\Tor(\fg)$ be a torus of maximal dimension.
Suppose that $S(\fg)\cong\GL_{\mu(\fg)}(\FF_p)$.
Let $M=\oplus_{\lambda\in\Lambda_M}M_{\lambda}$ be a restricted $\fg$-module.
then the following statements hold:
\begin{enumerate}
\item[(1)] $\Lambda_M=\{0\}$, or $\Lambda_M$ contains $\Lie_p(\ft,k)\setminus\{0\}$.
\item[(2)] All weight spaces of $M$ belonging to
nonzero weights have the same dimension.
\end{enumerate}
In particular, let $\fg=\fg_0\oplus\bigoplus_{\alpha\in\Phi}\fg_{\alpha}$ be the root
space decomposition of $\fg$ relative to $\ft$,
then $\Phi\cup\{0\}=\Lie_p(\ft,k)$, and all root spaces of $\fg$ have same dimension.
\end{Lemma}

\begin{Remarks}
(1). In \cite{BFS}, the authors use the condition that $\cT_\fg(k)$ is irreducible.
In fact, $\cT_\fg(k)$ is irreducible if and only if
$S(\fg)\cong\GL_{\mu(\fg)}(\FF_p)$ (see \cite[Theorem 4.1]{Fa04}).

(2). Let $\ft\in\Tor(\fg)$ be a torus of maximal dimension.
There is an isomorphism: $\Lie_p(\ft,k)\cong\Hom_{\FF_p}(\FF_p^{\mu(\fg)},\FF_p)$.
\end{Remarks}

Recently, Skryabin asked the following question (see \cite[Question 5]{Skr15}).

\begin{Question}\label{Skryabin question}
Let $\mathfrak{L}$ be a simple finite-dimensional Lie algebra over an algebraically closed field
of characteristic $p>2$.
Suppose that $\mathfrak{L}$ is not restricted and $\ft$ is any torus in $\mathfrak{L}_p$.
Is it always true that all nonzero elements of the group
$$\{\alpha\in\ft^{*};~\alpha(t^{[p]})=\alpha(t)^p~~\text{for all}~~t\in\ft\}$$
are roots of $\mathfrak{L}$, that is, $\mathfrak{L}$ has precisely $p^{\dim\ft}-1$ nonzero roots
with respect to $\ft$?
\end{Question}

For convenience, we let $\Lambda(\fL,\ft)$ be the set
of nonzero weights relative to $\ft$.
Hence, Skryabin's question is true is equivalent to say $\Lambda(\fL,\ft)=\Lie_p(\ft,k)\setminus\{0\}$ for any torus $\ft$ in $\fL_p$.

In fact, it suffices to consider the torus of maximal dimension.
\begin{Proposition}\label{proposition 1}
Suppose that Question \ref{Skryabin question} is ture for some tours $\ft_0\in\Tor(\fL_p)$.
Then it is true for all tori.
In particular, if $S(\fL_p)\cong\GL_{\mu(\fL_p)}(\FF_p)$, then Question \ref{Skryabin question} is true.
\end{Proposition}
\begin{proof}
It is clear that $\fL$ is a restricted $\fL_p$-module (cf. \cite [Page 66, Proposition 1.3]{SF}).
Let
$$\fL=C_{\fL}(\ft_0)\oplus\bigoplus_{\alpha\in\Lambda(\fL,\ft_0)}\fL_\alpha$$
be the weight space decomposition relative to $\ft_0$,
where $\Lambda_\fL$ is the set of nonzero weights.
By assumption, we have $\Lambda_\fL=\Lie_p(\ft_0,k)\setminus\{0\}$.

Thanks to \cite[Corollary 4.2]{Fa04}, we know that Question \ref{Skryabin question} is true for all tori of maximal dimension.

We claim that Question \ref{Skryabin question} is ture for any subtorus $\ft\subseteq\ft_0$. Take $\ft'\subseteq\ft$ such that $\ft_0=\ft\oplus\ft'$.
Let
$$\fL=C_{\fL}(\ft)\oplus\bigoplus_{\beta\in\Lambda(\fL,\ft)}\fL_\beta.$$
be the weight space decomposition with respect to $\ft$.
Also, it is easy to verify that
$$\fL_{\beta}=\bigoplus\limits_{\alpha\in\Lambda(\fL,\ft_0), \alpha|_\ft=\beta}\fL_{\alpha}.$$
Furthermore, for each nonzero weight $\beta$,
the number
$$\sharp\{\alpha\in\Lambda(\fL,\ft_0);~\alpha|_\ft=\beta\}=p^{\dim\ft'}.$$
Hence, we obtain $\Lambda(\fL,\ft)=\Lie_p(\ft,k)\setminus\{0\}$.

Now, it follows from Lemma \ref{weight space decomp} and the above arguments,
it is clear that Question \ref{Skryabin question} holds for
all maximal tori.
Last, let $\ft$ be an arbitrary torus.
Obviously, $\ft$ is contained in some maximal torus.
Now we only need to apply the same reasoning to $\ft$.

Suppose that $S(\fL_p)\cong\GL_{\mu(\fL_p)}(\FF_p)$.
Note that $\fL$ is simple, so that $\Lambda(\fL,\ft_0)\neq\emptyset$.
so that our assertion follows from Lemma \ref{BFSlemma}(1).
\end{proof}

\begin{Corollary}
Let $\fL$ be one of the following non-restricted simple Lie algebras
$$H(2;\underline{1};\Phi(\tau))^{(1)}, W(m,\underline{n})~(\underline{n}\neq\underline{1}).$$
Then $\Lambda(\fL,\ft)=\Lie_p(\ft,k)\setminus\{0\}$ for any torus $\ft$ in $\fL_p$.
\end{Corollary}
\begin{proof}
It follows from Lemma \ref{toralstablizerofh21phitau} and Theorem \ref{S(g)Zassenhaus} that there is an isomorphism $$S(\fL_p)\cong\GL_{\mu(\fL_p)}(\FF_p),$$
so that our statement follows from Proposition \ref{proposition 1}.
\end{proof}

\begin{Remarks}
It should be noticed that Question \ref{Skryabin question} always holds
whenever $\fL$ is a non-classical restricted simple Lie algebra.
In fact, by the classification (cf.\cite{BW}, \cite{S3}), we know $\fL$ is Cartan type or Melikian type.
Hence, the fact follows from Lemma \ref{S(g)cartan type}, Lemma \ref{S(g)melikian} and Lemma \ref{BFSlemma}.
\end{Remarks}

We next give a short proof of some known results (see \cite[Lemma 4.6.4]{BW82}, \cite[Theorem 7.6.2, Theorem 7.6.5]{S1}, \cite[Theorem 10.3.2]{S2}).

\begin{Corollary}[Block-Wilson, Strade]
Let $\fL$ be one of the following non-restricted simple Lie algebras
$$H(2;\underline{1};\Phi(\tau))^{(1)}, W(1,\underline{n}),$$
and $\ft\in\Tor(\fL_p)$ a torus of maximal dimension.
Then the following statements hold:
\begin{enumerate}
\item[(1)] $\ft$ is a self-centralizing torus of $\fL_p$, i.e., $C_{\fL_p}(\ft)=\ft$.
\item[(2)] $\dim_k\fL_\alpha=1$ for every $\alpha\in\Lambda(\fL,\ft)$.
\item[(3)] $\dim_k(W(1,\underline{n})\cap\ft)=1$.
\item[(4)] $\dim_k(H(2;\underline{1};\Phi(\tau))^{(1)}\cap\ft)=0$, i.e., $C_{H(2;\underline{1};\Phi(\tau))^{(1)}}(\ft)=\{0\}$.
\end{enumerate}
\end{Corollary}
\begin{proof}
Let
$$\fL=C_{\fL}(\ft)\oplus\bigoplus_{\alpha\in\Lambda(\fL,\ft)}\fL_\alpha.$$ be the weight space decomposition with respect to $\ft$.
According to Lemma \ref{toralstablizerofh21phitau}, Theorem \ref{S(g)Zassenhaus} and Lemma \ref{BFSlemma}(2),
we have
$$\dim_k\fL=\dim_kC_{\fL}(\ft)+(p^{\mu(\fg)}-1)\dim_k\fL_{\alpha}.$$
Similarly, we also have
$$\dim_k\fg=\dim_kC_{\fg}(\ft)+(p^{\mu(\fg)}-1)\dim_k\fg_{\alpha},$$
where $\fg:=\fL_p$.
Note that the facts $$\dim_kW(1,\underline{n})=p^n,~\dim_kH(2;\underline{1};\Phi(\tau))^{(1)}=p^2-1$$ and $$\dim_kW(1,\underline{n})_p=p^n+n-1,~\dim_kH(2;\underline{1};\Phi(\tau))_p^{(1)}=p^2+1,$$
so that our assertions follow from a basic computation.
\end{proof}


\end{document}